\documentclass[final,leqno]{siamltex}

\usepackage{hyperref}

\usepackage{amsmath,amssymb,amsfonts,mathrsfs,amsthm}
\usepackage[titletoc,toc,title]{appendix}

\usepackage{array}
\usepackage[utf8]{inputenc}
\usepackage{listings}
\usepackage{mathtools}
\usepackage{pdfpages}
\usepackage[textsize=footnotesize,color=green]{todonotes}
\usepackage{bm}
\usepackage[normalem]{ulem}
\usepackage{hhline}

\usepackage{graphicx}
\usepackage{subfig}

\usepackage{color}

\newcommand{\pd}[2]{\frac{\partial#1}{\partial#2}}

\newcommand{\mc}[1]{\mathcal{#1}}

\newcommand{\LRp}[1]{\left( #1 \right)}

\newtheorem{lemma}[theorem]{Lemma}
\newtheorem{corollary}[theorem]{Corollary}

\newcommand{\eval}[2][\right]{\relax
  \ifx#1\right\relax \left.\fi#2#1\rvert}

\newcolumntype{C}[1]{>{\centering\let\newline\\\arraybackslash\hspace{0pt}}m{#1}}

\newcommand*\diff[1]{\mathop{}\!{\mathrm{d}#1}}

\makeatletter
\renewcommand\d[1]{\mspace{6mu}\mathrm{d}#1\@ifnextchar\d{\mspace{-3mu}}{}}
\makeatother

\author{Jesse Chan, T. Warburton}
\date{}
\title{A short note on a Bernstein-Bezier basis for the pyramid}
\begin{document}
\maketitle
\begin{abstract}
We introduce a Bernstein-Bezier basis for the pyramid, whose restriction to the face reduces to the Bernstein-Bezier basis on the triangle or quadrilateral.  The basis satisfies the standard positivity and partition of unity properties common to Bernstein polynomials, and spans the same space as non-polynomial pyramid bases in the literature \cite{bergot2010higher, nigam2012high, bergot2013higher, chan2015orthogonal}.  
\end{abstract}

\section{Introduction}

Bernstein-Bezier bases have long been ubiquitous in graphics and computer-aided design \cite{farin1986triangular, farin2014curves}, though they have recently received attention for their utility in the meshing of curved geometries and the numerical solution of partial differential equations \cite{hindenlang2010unstructured, johnen2014geometrical, geuzaine2015generation,michoski2015foundations}.  Ainsworth and Kirby both noticed that the structure of Bernstein-Bezier polynomials on simplices allowed for linear-complextiy algorithms for the assembly and multiplication of finite element matrices \cite{kirby2011fast, ainsworth2011bernstein, kirby2012fast}, as well as for linear-complexity Discontinuous Galerkin solvers \cite{kirby2015efficient}.  The hierarchical structure of Bernstein polynomials also facilitates the construction of structure-preserving vectorial basis functions \cite{kirby2014low, ainsworth2015bernstein}.  Additionally, their close connection to B-splines has been exploited to simplify the implementation of NURBS-based finite element methods \cite{borden2011isogeometric}.  

Bernstein bases on tensor product elements (quadrilaterals and hexahedra) enjoy a natural tensor-product construction, while the construction of Bernstein bases on simplices rely on barycentric coordinates to generalize the construction in one space dimension.  Bernstein-Bezier bases on the prism may similarly be expressed as the tensor product of triangular and one-dimensional basis functions.  However, Bernstein-Bezier bases for the pyramid (which acts as a transitional piece to couple together hexahedral and tetrahedral elements \cite{bergot2013higher, chan2015gpu}) have received less attention.  Recent work has considered the extension of such ideas to pyramids, but define Bernstein-Bezier pyramid bases (which are distinct from Bernstein-Bezier pyramid algorithms \cite{goldman2002pyramid, ainsworth2014pyramid}) by splitting the pyramid into two tetrahedra \cite{chen20113d,ainsworthpyr}.  The work presented here presents an alternative construction based on pyramid bases found in the literature \cite{bergot2010higher, nigam2012high, chan2015orthogonal}, which are non-polynomial, but contain the space of polynomials of total degree $N$ under vertex-based mappings.  Quadrature and the computation of matrices for finite element methods are also discussed, and condition numbers are presented for mass and stiffness matrices on the reference element.  

\section{Bernstein-Bezier bases for triangles and quadrilaterals}

Bernstein-Bezier basis functions on the triangle \cite{farin1986triangular} may be defined in terms of the barycentric coordinates $\lambda_1, \lambda_2, \lambda_3$
\[
B^N_{ijk} = C^N_{ijk} \lambda_1^i\lambda_2^j\lambda_3^k
\]
with $i+j+k = N$ and 
\[
C^N_{ijk} = \frac{N\,!}{i\,!j\,!k\,!}.
\]
For the unit right triangle with coordinates $(r,s)$, these barycentric coordinates are given explicitly as
\[
\lambda_1 = 1- (r+s), \quad \lambda_2 = r, \quad \lambda_3 = s.
\]
Assuming the Duffy transform
\[
r = a(1-b), \quad s = b,
\]
which maps the unit cube with coordinates $(a,b) \in [0,1]^2$ to the unit right triangle, we have that
\[
\lambda_1 = (1-a)(1-b), \quad \lambda_2 = a(1-b), \quad \lambda_3 = b.
\]
This gives a definition for the Bezier triangle basis on the quad
\[
B^N_{ijk} = C^N_{ijk} a^j(1-a)^ib^k(1-b)^{i+j}.
\]
Since $N = i + j + k$, $i+j = N-k$, and we may rewrite the above as
\begin{align*}
B^N_{ijk}(a,b) &= C^N_{ijk} a^j(1-a)^i b^k(1-b)^{N-k}\\
&= \frac{(N-k)\,!}{i\,!j\,!} a^j(1-a)^i B^N_k(b) \\
&= B^{N-k}_i(a)B^N_k(b) = B^{N-k}_j(a)B^N_k(b)
\end{align*}
where 
\[
B^N_k(b) = \binom{N}{k} b^k(1-b)^{N-k},
\]
is the standard order $N$ 1D Bernstein polynomial and $B^{N-k}_i(a)$ is the $i$th Bernstein polynomial of order $N-k$.  This decomposition is at the heart of the sum-factorization techniques used in \cite{kirby2011fast, ainsworth2011bernstein, kirby2012fast}.  Since the same barycentric approach is used to define Bernstein polynomials on a simplex of arbitrary dimension, the restriction of tetrahedral Bernstein-Bezier basis functions to a triangular face results in the triangular Bernstein-Bezier basis.  

Bernstein-Bezier basis functions on quadrilaterals or hexahedra are defined using a tensor product construction; for example, on a reference quadrilateral with coordinates $(a,b)$, 
\[
B^N_{ij}(a,b) = B^N_i(a)B^N_j(b).  
\]
Similarly to the tetrahedron, restricting hexahedral Bernstein-Bezier functions to a single face results in Bernstein-Bezier polynomials over the quadrilateral.

\section{A Bernstein-Bezier basis for the pyramid}

To extend this construction to the pyramid, we use a collapsed coordinate system \cite{karniadakis1999spectral}.  We define the unit cube with coordinates $(a,b,c)\in [0,1]^3$, such that the unit cube is mapped to the unit right pyramid through the transform
\[
r = a(1-c), \quad s = b(1-c), \quad t = c.
\]
A Bernstein-Bezier basis for the pyramid may be defined over the unit cube as follows
\[
B_{ijk}(a,b,c) = B^{N-k}_i(a)B^{N-k}_j(b)B^{N}_k(c).
\]
where the indices obey
\begin{align*}
0 &\leq k\leq N, \quad 0 \leq i,j \leq N-k.
\end{align*}
The total dimension of this space is $N_p = (N+1)(N+2)(2N+3)/6$.  
  
This construction is exactly what results from combining triangle Bernstein Bezier basis functions in the $a,c$ and $b,c$ coordinates.  The Bernstein pyramid basis above satisfies the standard positivity and partition of unity properties, and the traces reduce to Bernstein-Bezier basis functions on the triangle or quadrilateral.  This latter property simplifies the enforcement of conformity between pyramids and tetrahedral or hexahedral elements.  Additionally, the Bernstein pyramid space spans the same space as the rational basis of Bergot, Cohen, and Durufle \cite{bergot2010higher}, and thus contains polynomials of total degree $N$ under a vertex-based mapping of the unit right pyramid to physical space.  

\begin{lemma}
The Bernstein pyramid basis satisfies a pointwise positivity and partition of unity property on the pyramid. 
\end{lemma}
\begin{proof}
Since each component of the Bernstein pyramid basis is pointwise positive on the cube, the basis is pointwise positive over  the pyramid.  We may also show that the partition of unity property for the Bernstein pyramid is preserved: 
\begin{align*}
\sum_{k=0}^N \sum_{i=0}^{N-k}\sum_{j=0}^{N-k} B_{ijk} &= \sum_{k=0}^N B_k(c) \sum_{i=0}^{N-k} B^{N-k}_i(a) \sum_{j=0}^{N-k} B^{N-k}_j(b) = 1
\end{align*}
by the fact that $B^N_k$, $B^{N-k}_i,B^{N-k}_j$ all satisfy a partition of unity property over $[0,1]$.
\end{proof}

\begin{lemma}
The trace spaces of the Bernstein pyramid are Bernstein polynomials on the faces.   
\end{lemma}
\begin{proof}
We may prove the first property on the unit cube by restricting $a,b$ to either $0$ or $1$ for the triangular faces, and restricting $c = 0$ for the quadrilateral face.  For the quadrilateral face, since $B_k(c) = 0$ for $k>0$, the Bernstein basis is nonzero only if $k=0$. 
\[
B_{ij0} = B_i^N(a)B_j^N(b)
\]
which is exactly the tensor product Bernstein basis over the quadrilateral face.  For the triangular faces, we take $a,b = 0,1$.  We show the trace space for $a = 0$; the other faces are similar.  If $a = 0$, the Bernstein basis is nonzero only for $i = 0$, and the nonzero Bernstein basis functions follow the form
\[
B_{0jk} = B_j^{N-k}(b)B^N_k(c).
\]
This is identical to the Bernstein basis on the triangle after a mapping to the unit quadrilateral.  
\end{proof}

\begin{lemma}
The Bernstein pyramid space is identical to the pyramid space of Bergot, Cohen, Durufle \cite{bergot2010higher}.  
\end{lemma} 
\begin{proof}
This is simplest to show by showing equivalence with the pyramid basis presented in \cite{chan2015orthogonal}, which is also equivalent to the basis of Bergot, Cohen, and Durufle.  
\[
\phi_{ijk}(a,b,c) = \ell^k_i(a)\ell^k_j(b) \LRp{\frac{1-c}{2}}^k P^{2k+3}_{N-k}(c),
\]
or equivalently
\[
\ell^{N-k}_i(a)\ell^{N-k}_j(b) \LRp{\frac{1-c}{2}}^{N-k} P^{2(N-k)+3}_{k}(c).
\]
As the spaces spanned by each basis are of the same dimension, it simply remains to show that their span is identical.  Since the functions in $a,b$
\[
B^{N-k}_i(a)B^{N-k}_j(b), \qquad \ell^{N-k}_i(a)\ell^{N-k}_j(b), \qquad 0\leq i,j \leq N-k
\]  
both span $Q_{N-k}(a,b)$, and the functions in $c$
\[
\LRp{\frac{1-c}{2}}^{N-k} P^{2(N-k)+3}_{k}(c), \qquad B_k^{N}(c)
\]
are both linearly independent homogeneous polynomials of total order $N$, the two bases span the same approximation space.  

\end{proof}
As a result, the Bernstein pyramid contains the space of polynomials $P^N$ on any vertex-mapped pyramid, and high order accurate approximations may be constructed on general vertex-mapped pyramids.  

Unlike the orthogonal bases for pyramids constructed previously, the Bernstein basis absorbs the extra homogenizing factors of $c^k$ or $(1-c)^k$ into the definition of the Bernstein-Bezier basis in the $c$ direction, resulting in a very concise formula.  

\subsection{Derivatives, quadrature, and evaluation of matrices}

To compute derivatives with respect to reference coordinates $r,s,t$, we use the chain rule, involving factors of
\begin{align*}
\pd{a}{r} &= \pd{b}{s} = \frac{1}{1-c}, \qquad \pd{a}{t} = \frac{r}{(1-t)^2} = \frac{a}{1-c}\\
\pd{b}{t} &= \frac{s}{(1-t)^2}= \frac{b}{1-c}.
\end{align*}
If $k=N$, the derivative is zero.  For $k<N$, we have
\begin{align*}
\pd{B^N_{ijk}}{r} &= \pd{B^{N-k}_i(a)}{a}B^{N-k}_j(b) \frac{B^N_k(c)}{1-c}\\
\pd{B^N_{ijk}}{s} &= B^{N-k}_i(a)\pd{B^{N-k}_j(b)}{b} \frac{B^N_k(c)}{1-c}\\
\pd{B^N_{ijk}}{t} &= \pd{B^N_{ijk}}{a}\LRp{\frac{a}{1-c}} + \pd{B^N_{ijk}}{b}\LRp{\frac{b}{1-c}} + \pd{B^N_{ijk}}{c}.
\end{align*}
The additional factors of $a,b$ and $1/(1-c)$ may also be absorbed into the Bernstein polynomials, and properties of one-dimensional Bernstein-Bezier polynomials may be used to rewrite all derivatives in terms of Bernstein polynomials.  However, doing so does not readily simplify expressions for derivatives, which do not appear to be as sparse or concise for the pyramid as they are for the tetrahedron.  

Under the collapsed coordinate system, integrals on a physical pyramid $\mc{P}$ are computed via the transformation
\[
\int_{\mc{P}} u  \diff x \diff y \diff z = \int_{\widehat{\mc{P}}} u J \diff r \diff s \diff t= \int_{0}^1 \int_{0}^1 \int_{0}^1 u J(1-c)^2 \diff a \diff b \diff c.  
\]
where $J$ is the determinant of the Jacobian of the mapping from reference pyramid $\widehat{\mc{P}}$ to  $\mc{P}$.  For vertex-mapped pyramids, $J$, as well as all change of variables factors $\pd{rst}{xyz}$, are bilinear in the $a,b$ coordinates and constant in $c$ \cite[Lemma 3.5]{bergot2010higher}.  An appropriate quadrature for the pyramid may then be constructed using the tensor product of one-dimensional Gaussian quadratures in the $a, b$ coordinates and Gauss-Jacobi quadratures with weights $(2,0)$ in the $c$ coordinate, which integrates exactly the above expression.  

The entries of the mass matrix 
\begin{align*}
&M_{ijk,lmn} = \int_{\mc{P}} \phi_{ijk}\phi_{lmn}   = \int_{\widehat{\mc{P}}} \phi_{ijk}\phi_{lmn} J\\
&= \int_{a} B^{N-k}_i(a) B^{N-k}_l(a) \int_b  B^{N-k}_j(b) B^{N-k}_m(b) \int_c  B^{N}_k(c) B^{N}_n(c) (1-c)^2 J
\end{align*}
are straightforward to compute.  Fast algorithms for Bernstein-Bezier bases \cite{kirby2011fast, ainsworth2011bernstein, kirby2012fast} --- which use that the product of two Bernstein polynomials is again a scaled Bernstein polynomial, and that the integrals of Bernstein polynomials are explicitly known in terms of binomial coefficients --- may also be applied to the evaluation of the pyramidal Bernstein-Bezier mass matrix, though unlike vertex-mapped tetrahedron, the non-constant factor $J$ must be taken into account.  The weak derivative matrices $S^x, S^y, S^z$
\begin{align*}
S^x_{ijk,lmn} = \int_{\mc{P}} B^N_{ijk}\pd{B^N_{lmn}}{x}\\
S^y_{ijk,lmn} = \int_{\mc{P}} B^N_{ijk}\pd{B^N_{lmn}}{y}\\
S^z_{ijk,lmn} = \int_{\mc{P}} B^N_{ijk}\pd{B^N_{lmn}}{z}
\end{align*}
are also computed exactly using the quadrature rule described above for vertex-mapped pyramids \cite{chan2015orthogonal}.  However, the stiffness matrix is not integrated exactly using quadrature unless the transformation is affine (i.e. the base is a planar parallelogram) due to rational factors in the resulting integrals \cite{bergot2010higher}.  A non-conforming error analysis and numerical experiments appear to indicate that high order accuracy may still be attained using the quadrature rule described above to compute the stiffness matrix \cite{bergot2013higher}.  

\subsection{Conditioning of matrices}

Figure~\ref{fig:condnums} reports computed condition numbers of the reference mass and stiffness matrices for both pyramidal and tetrahedral Bernstein-Bezier bases at various orders of approximation.  Dirichlet boundary conditions are enforced to ensure that the stiffness matrix is nonsingular.  The condition number of each matrix grows exponentially in the order $N$, as expected for Bernstein-Bezier type bases.  However, the growth of the condition numbers for the pyramid is more rapid than the growth of the condition number for the tetrahedra, and may need to be addressed in order to maintain numerical accuracy at high orders of approximation.  

\begin{figure}[!h]
\centering
\subfloat[Mass matrix]{\includegraphics[width=.45\textwidth]{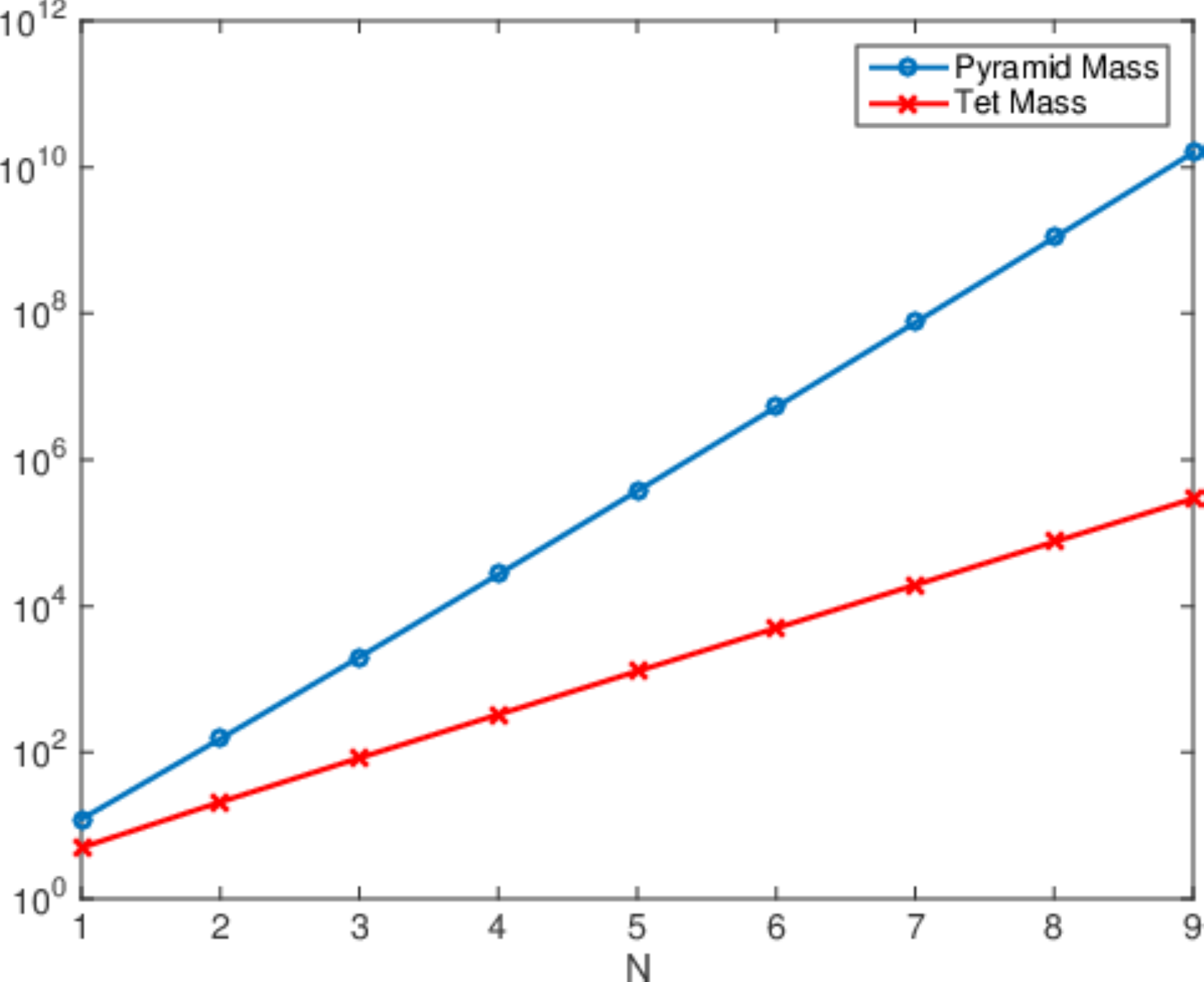}}
\hspace{1em}
\subfloat[Stiffness matrix]{\includegraphics[width=.45\textwidth]{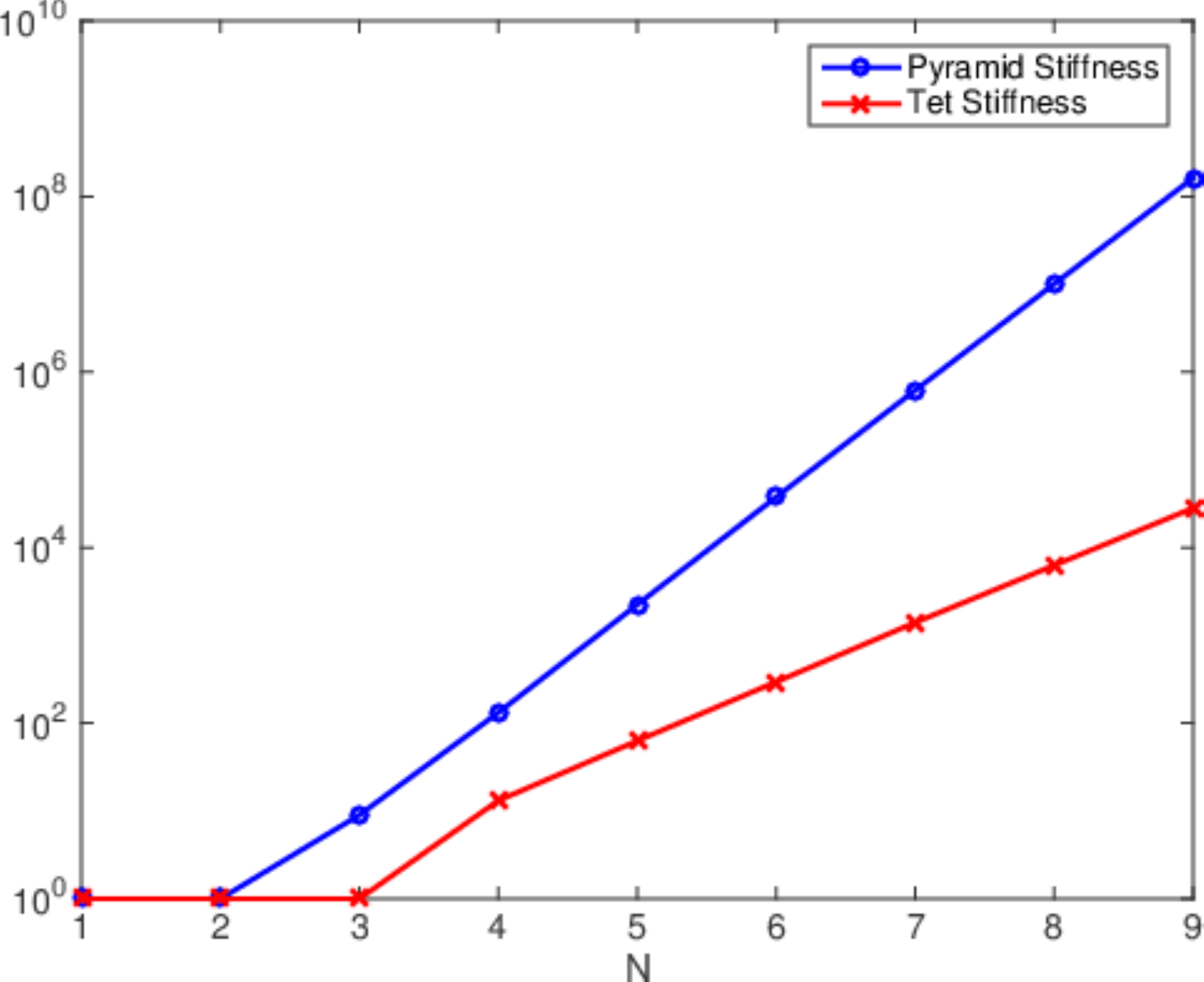}}
\caption{Condition numbers of reference tetrahedral and pyramidal Bernstein-Bezier mass and stiffness matrices at various orders of approximation $N$.}
\label{fig:condnums}
\end{figure}

\section{Acknowledgements}

JC would like to thank John Evans for informative discussions and for posing the question of how to construct a Bernstein-Bezier basis on the pyramid.  Both authors would like to acknowledge the support of NSF (award number DMS-1216674) in this research.  



\bibliographystyle{plain}
\bibliography{bernpyr}

\begin{thebibliography}{10}

\bibitem{ainsworth2014pyramid}
Mark Ainsworth.
\newblock Pyramid algorithms for {Bernstein--B{\'e}zier} finite elements of
  high, nonuniform order in any dimension.
\newblock {\em SIAM Journal on Scientific Computing}, 36(2):A543--A569, 2014.

\bibitem{ainsworth2011bernstein}
Mark Ainsworth, Gaelle Andriamaro, and Oleg Davydov.
\newblock {Bernstein-B{\'e}zier} finite elements of arbitrary order and optimal
  assembly procedures.
\newblock {\em SIAM Journal on Scientific Computing}, 33(6):3087--3109, 2011.

\bibitem{ainsworth2015bernstein}
Mark Ainsworth, Gaelle Andriamaro, and Oleg Davydov.
\newblock A {Bernstein--B{\'e}zier Basis} for arbitrary order {Raviart-Thomas}
  finite elements.
\newblock {\em Constructive Approximation}, 41(1):1--22, 2015.

\bibitem{ainsworthpyr}
Mark Ainsworth, Oleg Davydov, and Larry~L. Schumaker.
\newblock {Bernstein-Bezier} finite elements on
  tetrahedral-hexahedral-pyramidal partitions.
\newblock 2015.

\bibitem{bergot2010higher}
Morgane Bergot, Gary Cohen, and Marc Durufl{\'e}.
\newblock Higher-order finite elements for hybrid meshes using new nodal
  pyramidal elements.
\newblock {\em Journal of Scientific Computing}, 42(3):345--381, 2010.

\bibitem{bergot2013higher}
Morgane Bergot and Marc Durufl{\'e}.
\newblock Higher-order discontinuous {Galerkin} method for pyramidal elements
  using orthogonal bases.
\newblock {\em Numerical Methods for Partial Differential Equations},
  29(1):144--169, 2013.

\bibitem{borden2011isogeometric}
Michael~J Borden, Michael~A Scott, John~A Evans, and Thomas~JR Hughes.
\newblock Isogeometric finite element data structures based on {B{\'e}zier}
  extraction of {NURBS}.
\newblock {\em International Journal for Numerical Methods in Engineering},
  87(1-5):15--47, 2011.

\bibitem{chan2015gpu}
Jesse Chan, Zheng Wang, Axel Modave, Jean-Francois Remacle, and T~Warburton.
\newblock {GPU}-accelerated discontinuous {G}alerkin methods on hybrid meshes.
\newblock {\em arXiv preprint arXiv:1507.02557}, 2015.

\bibitem{chan2015orthogonal}
Jesse Chan and T~Warburton.
\newblock Orthogonal bases for vertex-mapped pyramids.
\newblock {\em arXiv preprint arXiv:1502.07703}, 2015.

\bibitem{chen20113d}
Juan Chen, Chong-Jun Li, and Wan-Ji Chen.
\newblock A {3D} pyramid spline element.
\newblock {\em Acta Mechanica Sinica}, 27(6):986--993, 2011.

\bibitem{farin1986triangular}
Gerald Farin.
\newblock Triangular {Bernstein-B{\'e}zier} patches.
\newblock {\em Computer Aided Geometric Design}, 3(2):83--127, 1986.

\bibitem{farin2014curves}
Gerald Farin.
\newblock {\em Curves and surfaces for computer-aided geometric design: a
  practical guide}.
\newblock Elsevier, 2014.

\bibitem{geuzaine2015generation}
Christophe Geuzaine, Amaury Johnen, Jonathan Lambrechts, J-F Remacle, and
  Thomas Toulorge.
\newblock The generation of valid curvilinear meshes.
\newblock In {\em IDIHOM: Industrialization of High-Order Methods-A Top-Down
  Approach}, pages 15--39. Springer, 2015.

\bibitem{goldman2002pyramid}
Ron Goldman.
\newblock {\em Pyramid algorithms: A dynamic programming approach to curves and
  surfaces for geometric modeling}.
\newblock Morgan Kaufmann, 2002.

\bibitem{hindenlang2010unstructured}
F~Hindenlang, G~Gassner, T~Bolemann, and CD~Munz.
\newblock Unstructured high order grids and their application in discontinuous
  {G}alerkin methods.
\newblock In {\em Conference Proceedings, V European Conference on
  Computational Fluid Dynamics ECCOMAS CFD}, pages 1--8, 2010.

\bibitem{johnen2014geometrical}
Amaury Johnen, J-F Remacle, and Christophe Geuzaine.
\newblock Geometrical validity of high-order triangular finite elements.
\newblock {\em Engineering with Computers}, 30(3):375--382, 2014.

\bibitem{karniadakis1999spectral}
George Karniadakis and Spencer~J Sherwin.
\newblock {\em Spectral/hp Element Methods for {CFD}}.
\newblock Oxford University Press, 1999.

\bibitem{kirby2011fast}
Robert~C Kirby.
\newblock Fast simplicial finite element algorithms using {Bernstein}
  polynomials.
\newblock {\em Numerische Mathematik}, 117(4):631--652, 2011.

\bibitem{kirby2014low}
Robert~C Kirby.
\newblock Low-complexity finite element algorithms for the de {Rham} complex on
  simplices.
\newblock {\em SIAM Journal on Scientific Computing}, 36(2):A846--A868, 2014.

\bibitem{kirby2015efficient}
Robert~C Kirby.
\newblock Efficient discontinuous {Galerkin} finite element methods via
  bernstein polynomials.
\newblock {\em arXiv preprint arXiv:1504.03990}, 2015.

\bibitem{kirby2012fast}
Robert~C Kirby and Kieu~Tri Thinh.
\newblock Fast simplicial quadrature-based finite element operators using
  {Bernstein} polynomials.
\newblock {\em Numerische Mathematik}, 121(2):261--279, 2012.

\bibitem{michoski2015foundations}
C~Michoski, J~Chan, L~Engvall, and JA~Evans.
\newblock Foundations of the {B}lended {I}sogeometric {D}iscontinuous
  {G}alerkin {(BIDG)} method.
\newblock 2015.

\bibitem{nigam2012high}
Nilima Nigam and Joel Phillips.
\newblock High-order conforming finite elements on pyramids.
\newblock {\em IMA Journal of Numerical Analysis}, 32(2):448--483, 2012.

\end{thebibliography}

\end{document}